\documentclass[11pt,bezier]{article}
\usepackage{amsmath,amssymb,amsfonts,euscript,graphicx}

\newtheorem{prethm}{{\bf Theorem}}

\newenvironment{thm}{\begin{prethm}{\hspace{-0.5
               em}{\bf.}}}{\end{prethm}}

\newtheorem{prepro}[prethm]{{\bf Theorem}}

\newtheorem{preprop}[prethm]{{\bf Proposition}}

\newenvironment{prop}{\begin{preprop}{\hspace{-0.5
               em}{\bf.}}}{\end{preprop}}

\newtheorem{precor}[prethm]{{\bf Corollary}}

\newenvironment{cor}{\begin{precor}{\hspace{-0.5
               em}{\bf.}}}{\end{precor}}

\newtheorem{preconj}[prethm]{{\bf Conjecture}}

\newtheorem{preremark}[prethm]{{\bf Remark}}

\newenvironment{remark}{\begin{preremark}\rm{\hspace{-0.5
               em}{\bf.}}}{\end{preremark}}

\newtheorem{preexample}[prethm]{{\bf Example}}

\newenvironment{example}{\begin{preexample}\rm{\hspace{-0.5
               em}{\bf.}}}{\end{preexample}}

\newtheorem{prelem}[prethm]{{\bf Lemma}}

\newenvironment{lem}{\begin{prelem}{\hspace{-0.5
               em}{\bf.}}}{\end{prelem}}

\newtheorem{prelam}{{\bf Lemma}}

\newtheorem{preproof}{{\bf Proof.}}

\newenvironment{proof}[1]{\begin{preproof}{\rm
               #1}\hfill{$\Box$}}{\end{preproof}}

\title{\bf \large  Some Properties of the Nil-Graphs of\\
Ideals of Commutative Rings
\thanks
{{\it Key Words}: Nil-graph; Complete graph; Bipartite graph; Genus; Independence number.}
\thanks {2010{ \it Mathematics Subject Classification}: 05C15; 05C69; 13E05; 13E10.}}
\author{{\normalsize {\sc ${}^{\mathsf{}}$}  {\sc R. Nikandish${}^{\mathsf{a}}$\thanks{Corresponding author}}, {\sc F. Shaveisi${}^{\mathsf{b}}$}} \\
 {\footnotesize{${}^{\mathsf{}}$}}\\
 {\footnotesize{${}^{\mathsf{a}}$\it Department of Basic Sciences, Jundi-Shapur University of Technology,
Dezful, Iran}}\\
{\footnotesize{${}^{\mathsf{}}$ \it {\rm P.O. Box 64615-334},}}\\
{\footnotesize{${}^{\mathsf{b}}$\it Department of Mathematics, Faculty of Sciences,
Razi University, Kermanshah, Iran
}}\\
{\footnotesize{}}\\
{\footnotesize{
$\mathsf{}$\quad\quad$\mathsf{r.nikandish@ipm.ir}$\quad\quad$\mathsf{f.shaveisi@ipm.ir}$}}}

\date{}

\begin{document}
\maketitle
\begin{abstract}
{\small Let $R$ be a commutative ring with identity and
${\rm Nil}(R)$ be the set of nilpotent elements of $R$. The nil-graph
of ideals of $R$ is defined as the graph $\mathbb{AG}_N(R)$ whose
vertex set is $\{I:\ (0)\neq I\lhd R$ and there exists a non-trivial ideal $J$ such
that $IJ\subseteq {\rm Nil}(R)\}$ and two distinct vertices $I$ and $J$
are adjacent if and only if $IJ\subseteq {\rm Nil}(R)$. Here, we study conditions under which $\mathbb{AG}_N(R)$ is complete or
bipartite. Also, the independence number of $\mathbb{AG}_N(R)$ is determined, where $R$ is
a reduced ring. Finally, we classify Artinian rings whose nil-graphs of ideals have
genus at most one.
}
\end{abstract}

\vspace{9mm} \noindent{\bf\large 1. Introduction}

When one
assigns a graph with an algebraic structure, numerous interesting algebraic problems arise from
the translation of some graph-theoretic parameters such as clique number, chromatic number,
diameter, radius and so on. There are many papers in this topic, see for example \cite{MBI}, \cite{maimani} and \cite{nilgraph1}.
Throughout this paper,
all rings are assumed to be non-domain commutative rings with
identity. By $\mathbb{I}(R)$ ($\mathbb{I}(R)^*$) and ${\rm Nil}(R)$, we
denote the set of all proper (non-trivial) ideals of $R$ and the nil-radical of $R$, respectively. The set of all maximal and minimal
prime ideals of $R$ are denoted by ${\rm Max}(R)$ and ${\rm Min}(R)$, respectively. The ring $R$
is said to be \textit{reduced}, if it has no non-zero nilpotent
element.

 Let $G$ be a graph. The degree of a vertex $x$ of $G$ is denoted by $d(x)$. The
graph $G$ is said to be \textit{$r$-regular}, if the degree of
each vertex is $r$.
The \textit{complete graph} with $n$ vertices, denoted by
$K_n$, is a graph in which any two distinct vertices are
adjacent. A \textit{bipartite graph} is a graph whose vertices can
be divided into two disjoint parts $U$ and $V$ such that every
edge joins a vertex in $U$ to one in $V$. It is well-known that a
bipartite graph is a graph that does not contain any odd cycle. A
\textit{complete bipartite graph} is a bipartite graph in which
every vertex of one part is joined to every vertex of the other
part. If the size of one of the parts is $1$, then it is said to
be a \textit{star graph}. A \textit{tree} is a connected graph without cycles. Let $S_k$ denote the sphere with $k$ handles, where $k$ is a non-negative integer, that
is, $S_k$ is an oriented surface of genus $k$. The \textit{genus} of a graph $G$, denoted by
$\gamma(G)$, is the minimal integer $n$ such that the graph can be embedded in $S_n$. A genus 0 graph is called a \textit{planar graph}. It is well-known that
$$\gamma(K_n)=\lceil\frac{(n-3)(n-4)}{12}\rceil {\rm ~for~all~} n\geq 3,$$

$$\gamma(K_{m,n})=\lceil\frac{(m-2)(n-2)}{4}\rceil, {\rm ~for~all~} n\geq 2~{\rm and}~m\geq 2.$$
For a
graph $G$, 
the \textit{independence number} of $G$
is denoted by $\alpha(G)$. For more details about the used
terminology of graphs, see \cite{west}.

 We denote the annihilator of an ideal $I$ by $Ann(I)$. Also, the ideal $I$ of $R$ is called an \textit{annihilating-ideal} if $Ann(I)\neq (0)$. The notation $\mathbb{A}(R)$ is used for the set of all
annihilating-ideals of $R$. By the \textit{annihilating-ideal
graph} of $R$, $\mathbb{AG}(R)$, we mean the graph with vertex
set $\mathbb{A}(R)^{*}= \mathbb{A}(R)\setminus \{0\}$ and two
distinct vertices $I$ and $J$ are adjacent if and only if $IJ=0$.
Some properties of this graph have been studied in
\cite{disc.math,alal,behboodigenus,MBI,MBII}. In \cite{nilgraph1}, the authors have introduced another kind
of graph, called the nil-graph of ideals. The nil-graph
of ideals of $R$ is defined as the graph $\mathbb{AG}_N(R)$ whose
vertex set is $\{I:\ (0)\neq I\lhd R$ and there exists a non-trivial ideal $J$ such
that $IJ\subseteq {\rm Nil}(R)\}$ and two distinct vertices $I$ and $J$
are adjacent if and only if $IJ\subseteq {\rm Nil}(R)$.
Obviously, our definition is slightly
different from the one defined by Behboodi and Rakeei in
\cite{MBI} and it is easy to see that the usual
annihilating-ideal graph $\mathbb{AG}(R)$ is a subgraph of
$\mathbb{AG}_N(R)$. In \cite{nilgraph1}, some basic properties of nil-graph of ideals have been studied. In this
article, we continue the study of the nil-graph of ideals. In Section 2, the necessary and sufficient conditions, under which the nil-graph of a ring is complete or bipartite, are found. Section 3 is devoted to the studying of independent sets in nil-graph ideals. In Section 4, we classify all Artinian rings whose nil-graphs of ideals have genus at most one.

\vspace{5mm} \noindent{\bf\large 2. When Is the Nil-Graph of Ideals Complete or Bipartite?}\\

 In this section, we study conditions under which the nil-graph of ideals of a commutative ring is complete or complete bipartite. For instance, we show that if $R$ is a Noetherian ring, then $\mathbb{AG}_N(R)$ is a
complete graph if and only if either $R$ is  Artinian local or
$R\cong F_1\times F_2$, where $F_1$ and $F_2$ are fields. Also, it is proved that if $\mathbb{AG}_N(R)$ is bipartite, then
$\mathbb{AG}_N(R)$ is complete bipartite. Moreover, if $R$ is
non-reduced, then $\mathbb{AG}_N(R)$ is star and ${\rm Nil}(R)$ is the
unique minimal prime ideal of $R$.

 We start with the following theorem which can be viewed as a consequence of \cite[Theorem 5]{nilgraph1} (Here we prove it independently). Note that it is clear that if $R$ is a reduced ring, then  $\mathbb{AG}_N(R)\cong \mathbb{AG}(R)$.

\begin{thm}\label{artiniancomplete}
Let $R$ be a Noetherian ring. Then $\mathbb{AG}_N(R)$ is a
complete graph if and only if either $R$ is  Artinian local or
$R\cong F_1\times F_2$, where $F_1$ and $F_2$ are fields.
\end{thm}
\begin{proof}
{First suppose that $\mathbb{AG}_N(R)$ is complete. If $R$ is
reduced, then by \cite[Theorem 2.7]{MBI}, $R\cong F_1\times F_2$.
Thus we can suppose that ${\rm Nil}(R)\neq (0)$. We continue the proof
in the following two cases:

Case 1. $R$ is a local ring with the unique maximal ideal
$\mathfrak{m}$. Since $R$ is non-reduced, by Nakayama's Lemma
(see \cite[Proposition 2.6]{ati}), $\mathfrak{m}$ and
$\mathfrak{m}^2$ are two distinct vertices of
$\mathbb{AG}_N(R)$. Thus $\mathfrak{m}^3\subseteq {\rm Nil}(R)$ and so
$R$ is an Artinian local ring.

Case 2. $R$ has at least two maximal ideals. First we show that
$R$ has exactly two maximal ideals. Suppose to the contrary,
$\mathfrak{m}, \mathfrak{n}$ and $\mathfrak{p}$ are three
distinct maximal ideals of $R$. Since $\mathbb{AG}_N(R)$ is
complete, we deduce that $\mathfrak{m} \mathfrak{n}\subseteq
{\rm Nil}(R)\subseteq \mathfrak{p}$, a contradiction. Thus $R$ has
exactly two maximal ideals, say $\mathfrak{m}$ and
$\mathfrak{p}$. Now, we claim that both $\mathfrak{m}$ and
$\mathfrak{p}$ are minimal prime ideals. Since $\mathfrak{m}$ and
$\mathfrak{p}$ are adjacent, we conclude one of the maximal
ideals, say $\mathfrak{p}$, is a minimal prime ideal of $R$. Now,
suppose to the contrary, $\mathfrak{m}$ properly contains a
minimal prime ideal $\mathfrak{q}$ of $R$. Since $\mathfrak{m}
\mathfrak{p}\subseteq \mathfrak{q}$, we get a contradiction. So
the claim is proved. Thus $R$ is Artinian. Hence by \cite[Theorem
8.7 ]{ati}, we have $R\cong R_1\times R_2$, where $R_1$ and $R_2$
are Artinian local rings. By contrary and with no loss of
generality, suppose that $R_1$ contains a non-trivial ideal, say
$I$. Then the vertices $I\times R_2$ and $(0)\times R_2$ are not
adjacent, a contradiction. Thus $R\cong F_1\times F_2$, where
$F_1$ and $F_2$ are fields.\\
Conversely, if $R\cong F_1\times F_2$, where $F_1$ and $F_2$ are
fields, then it is clear that $\mathbb{AG}_N(R)\cong K_2$. Now,
suppose that $(R,\mathfrak{m})$ is an Artinian local ring. Since
$\mathfrak{m}$ is nilpotent, it follows that $\mathbb{AG}_N(R)$
is complete.}
\end{proof}

The following example shows that Theorem \ref{artiniancomplete}
does not hold for non-Noetherian rings.
\begin{example}
Let $R=\frac{k[x_i:\, i\geq 1]}{(x_i^2:\, i\geq 1)}$, where $k$
is a field. Then $R$ is not Artinian and $\mathbb{AG}_N(R)$ is a
complete graph.
\end{example}

\begin{remark}\label{niladjacent}
Let $R$ be a ring. Every non-trivial ideal contained in ${\rm Nil}(R)$ is adjacent to
every other vertex of $\mathbb{AG}_N(R)$. In particular, if $R$ is an Artinian
local ring, then $\mathbb{AG}_N(R)$ is a complete graph.
\end{remark}

The next result shows that nil-graphs, whose every vertices have finite degrees, are finite graphs.
\begin{thm}\label{finitedegree}
If every vertex of $\mathbb{AG}_N(R)$ has a finite degree, then $R$ has finitely many ideals.
\end{thm}
\begin{proof}
{First suppose that $R$ is non-reduced. Since $d({\rm Nil}(R))<\infty$, the assertion follows from Remark \ref{niladjacent}. Thus we can assume that $R$ is reduced. Choose $0\neq x\in Z(R)$. Since $d(Rx)<\infty$ and $Rx$ is adjacent to every ideal contained in $Ann(x)$, we deduce that $Ann(x)$ is an Artinian $R$-module. Similarly, one can show that $Rx$ is an Artinian $R$-module. Now, the $R$-isomorphism $Rx\cong \frac{R}{Ann(x)}$ implies that $R$ is an Artinian ring. Now, since $R$ is reduced, \cite[Theorem 8.7]{ati} implies that $R$ is a direct product of finitely many fields and hence we are done.~}
\end{proof}

The next result gives another condition under which $\mathbb{AG}_N(R)$ is complete.
\begin{thm}\label{regular}
If $\mathbb{AG}_N(R)$ is an $r$-regular graph, then $\mathbb{AG}_N(R)$ is a complete graph.
\end{thm}
\begin{proof}
{If ${\rm Nil}(R)\neq (0)$, then by Remark \ref{niladjacent}, there is nothing to prove. So, suppose that $R$ is reduced. Since $\mathbb{AG}_N(R)$ is an $r$-regular graph, Theorem \ref{finitedegree} and \cite[Theorem 8.7]{ati} imply that $R\cong F_1\times\cdots\times F_n$, where $n\geq 2$ and each $F_i$ is a field. It is not hard to check that every ideal $I=I_1\times \cdots \times I_n$ of $R$ has degree $2^{n_I}-1$, where
$n_I=|\{i:\ 1\leq i\leq n\ {\rm and}\ I_i=(0)\}|$. Let $I=F_1\times (0)\times\cdots\times (0)$ and $J=F_1\times\cdots\times F_{n-1}\times(0)$. Then we have $d(I)=2^{n-1}-1$ and $d(J)=1$. The $r$-regularity of $\mathbb{AG}_N(R)$ implies that $2^{n-1}-1=1$ and so $n=2$. Therefore $R\cong F_1\times F_2$, as desired.}
\end{proof}

In the rest of this section,  we study bipartite nil-graphs of ideals of rings.
\begin{thm}\label{bipartite}
Let $R$ be a ring such that $\mathbb{AG}_N(R)$ is bipartite. Then
$\mathbb{AG}_N(R)$ is complete bipartite. Moreover, if $R$ is
non-reduced, then $\mathbb{AG}_N(R)$ is star and ${\rm Nil}(R)$ is the
unique minimal prime ideal of $R$.
\end{thm}
\begin{proof}
{If $R$ is reduced, then by \cite[Corollary 2.5]{MBII},
$\mathbb{AG}_N(R)$ is a complete bipartite graph. Now, suppose
that $R$ is non-reduced. Then by Remark \ref{niladjacent},
$\mathbb{AG}_N(R)$ is a star graph. So, by Remark
\ref{niladjacent}, either ${\rm Nil}(R)$ is a minimal ideal or $R$ has
exactly two ideals. In the latter case, $R$ is an Artinian local
ring and so ${\rm Nil}(R)$ is the unique minimal prime ideal of $R$.
Thus we can assume that ${\rm Nil}(R)=(x)$ is a minimal ideal of $R$,
for some $x\in R$. To complete the proof, we show that $R$ has
exactly one minimal prime ideal. Suppose to the contrary,
$\mathfrak{p}_1$ and $\mathfrak{p}_2$ are two distinct minimal
prime ideals of $R$. Choose $z\in \mathfrak{p}_1\setminus
\mathfrak{p}_2$ and set $S_1=R\setminus \mathfrak{p}_1$ and
$S_2=\{1,z,z^2,\ldots\}$. If $0\notin S_1S_2$, then by
\cite[Theorem 3.44]{sharp}, there exists a prime ideal
$\mathfrak{p}$ in $R$ such that $\mathfrak{p}\cap
S_1S_2=\varnothing$ and hence $\mathfrak{p}= \mathfrak{p}_1$, a
contradiction. So, $0\in S_1S_2$. Therefore, there exist positive
integer $k$ and $y\in R\setminus \mathfrak{p}_1$ such that
$yz^k=0$. Consider the ideals $(x),
(y)$ and $(z^k)$. This is clear that $(x),(y)$ and $(z^k)$ are
three distinct vertices which form a triangle in
$\mathbb{AG}_N(R)$, a contradiction.}
\end{proof}

The following corollary is an immediate consequence of Theorem \ref{bipartite} and Remark \ref{niladjacent}.
\begin{cor}
If $\mathbb{AG}_N(R)$ is a tree, then $\mathbb{AG}_N(R)$ is a star
graph.
\end{cor}

We finish this section with the next corollary.
\begin{cor}
Let $R$ be an Artinian ring. Then $\mathbb{AG}_N(R)$ is bipartite if and only if
$\mathbb{AG}_N(R)\cong K_n$, where $n\in\{1,2\}$.
\end{cor}
\begin{proof}
{Let $R$ be an Artinian ring and $\mathbb{AG}_N(R)$ be bipartite. Then by Theorem \ref{bipartite}, $\mathbb{AG}_N(R)$ is complete bipartite.  If $R$ is local, then Remark \ref{niladjacent} implies that $\mathbb{AG}_N(R)$ is complete. Since $\mathbb{AG}_N(R)$ is complete bipartite, we deduce that $\mathbb{AG}_N(R)\cong K_n$, where $n\in\{1,2\}$. Now, suppose that $R$ is not local. Then by \cite[Theorem 8.7]{ati}, there exists a positive integer $n$ such that $R\cong R_1 \times\cdots\times R_n$, where every $R_i$ is an Artinian local ring. Since $\mathbb{AG}_N(R)$ contains no odd cycle, it follows that $n=2$. To complete the proof, we show that both $R_1$ and $R_2$ are fields. By contrary and with no loss of generality, suppose that $R_1$ contains a non-trivial ideal, say $I$. Then it is not hard to check that $R_1\times (0), I\times (0)$ and $(0)\times R_2$ forms a triangle in $\mathbb{AG}_N(R)$, a contradiction. The converse is trivial.}
\end{proof}


\vspace{5mm} \noindent{\bf\large 3. The Independence Number of Nil-Graphs of Ideals}\\

 In this section, we use the maximal intersecting families to obtain a low bound for the independence number of nil-graphs of ideals.
Let 
$R\cong R_1\times R_2\times\cdots\times R_n$,
 $$\mathcal{T}(R)=\{(0)\neq I= I_1\times I_2\times\cdots\times I_n\vartriangleleft R|\ \forall\ 1\leq k\leq n:\  I_k\in\{(0), R_k\}\};$$
and denote the induced subgraph of $\mathbb{AG}_N(R)$ on $\mathcal{T}(R)$ by $G_{\mathcal{T}}(R)$.
\begin{prop}\label{alphaGS}
If $R\cong R_1\times R_2\times\cdots\times R_n$ is a ring, then $\alpha(G_{\mathcal{T}}(R))=2^{n-1}$.
\end{prop}
\begin{proof}
{For every ideal $I= I_1\times I_2\times \cdots \times I_n$, let
$$\Delta_I=\{k|\ 1\leq k\leq n \ {\rm and}\ I_k=R_k\};$$
Then two distinct vertices $I$ and $J$ in $G_{\mathcal{T}}(R)$ are not adjacent if and only if  $\Delta_I\cap \Delta_J\neq\varnothing$. So, there is a one to one correspondence between the independent sets of $G_{\mathcal{T}}(R)$ and the set of
families of pairwise intersecting subsets of the set $[n]=\{1,2,\ldots, n\}$.
%
%
So,  \cite[Lemma 2.1]{maximalintersecting} completes the proof.
}
\end{proof}

Using \cite[Theorem 8.7]{ati}, we have the following immediate corollary.

\begin{cor}\label{artinianalpha}
Let $R$ be an Artinian with $n$ maximal ideals.  Then $\alpha(\mathbb{AG}_N(R))\geq 2^{n-1}$; moreover, the equality holds if and only if $R$ is reduced.
\end{cor}

\begin{lem}\label{matlis} {\rm\cite[Proposition 1.5]{matlis}}
Let $R$ be a ring and $\{{\mathfrak{p}}_1,\ldots,
{\mathfrak{p}}_n\}$ be a finite set of distinct minimal prime ideals
of $R$. Let $S=R\setminus \bigcup_{i=1}^{n}{\mathfrak{p}}_i$. Then
$R_S \cong R_{{\mathfrak{p}}_1}\times\cdots\times
R_{{\mathfrak{p}}_n}$.
\end{lem}

\begin{prop}\label{cormatlis}
If $|{\rm Min}(R)|\geq n$, then $\alpha(\mathbb{AG}_N(R))\geq 2^{n-1}$.
\end{prop}
\begin{proof}
{Let $\{{\mathfrak{p}}_1,\ldots,
{\mathfrak{p}}_n\}$ be a subset of ${\rm Min}(R)$ and $S=R\setminus \bigcup_{i=1}^{n}{\mathfrak{p}}_i$. By Lemma \ref{matlis},
there exists a ring isomorphism $R_S\cong R_{{\mathfrak{p}}_1}\times\cdots\times
R_{{\mathfrak{p}}_n}$.
On the other hand, if $I_S,J_S$ are two non-adjacent vertices of $\mathbb{AG}_N(R_S)$, then it is not hard to check that  $I,J$ are two non-adjacent vertices of $\mathbb{AG}_N(R)$. Thus $\alpha(\mathbb{AG}_N(R))\geq \alpha(\mathbb{AG}_N(R_S))$ and so by Proposition \ref{alphaGS}, we deduce that $\alpha(\mathbb{AG}_N(R))\geq 2^{n-1}$.
}
\end{proof}

From the previous proposition, we have the following immediate corollary which shows that the finiteness of $\alpha(\mathbb{AG}_N(R))$ implies the finiteness of number of the minimal prime ideals of $R$.
\begin{cor}\label{infiniteminclique}
If $R$ contains infinitely many minimal prime ideals, then the independence number of $\mathbb{AG}_N(R)$ is infinite.
\end{cor}

\begin{thm}\label{reducedfinitemin}
For every Noetherian reduced ring  $R$,  $\alpha(\mathbb{AG}_N(R))= 2^{|{\rm Min}(R)|-1}$.
\end{thm}
\begin{proof}
{Let ${\rm Min}(R)=\{\mathfrak{p}_1,\mathfrak{p}_2,\ldots, \mathfrak{p}_n\}$ and $S=R\setminus \bigcup_{k=1}^n\mathfrak{p}_k$. Then Lemma \ref{matlis} implies that
$
R_S\cong R_{{\mathfrak{p}}_1}\times\cdots\times
R_{{\mathfrak{p}}_n}$. On the other hand, by using  \cite[Proposition 1.1]{matlis}, we deduce that every $R_{\mathfrak{p}_i}$ is a field. Thus  $\alpha(\mathbb{AG}_N(R))\geq \alpha(\mathbb{AG}_N(R_S))=2^{n-1}$, by Corollary \ref{artinianalpha}. To complete the proof, it is enough to show that $\alpha(\mathbb{AG}_N(R))\leq \alpha(\mathbb{AG}_N(R_S))$. To see this, let $I(x_1,x_2,\ldots, x_r)$ and $J=(y_1,y_2,\ldots, y_s)$ be two non-adjacent vertices of $\mathbb{AG}_N(R)$. By \cite[Corollary 2.4]{huc}, $S$ contains no zero-divisor and so $I_S, J_S$ are non-trivial ideals of $R_S$. We show that $I_S, J_S$ are non-adjacent vertices of $\mathbb{AG}_N(R_S)$. Suppose to the contrary, $I_SJ_S\subseteq {\rm Nil}(R)_S=(0)$. Then for every $1\leq i\leq r$ and $1\leq j\leq s$, there exists $s_{ij}\in S$ such that $s_{ij}x_iy_j=0$. Setting $t=\prod_{i,j} s_{ij}$, we have $tIJ=(0)$. Since $t$ is not a zero-divisor, we deduce that $IJ=(0)$, a contradiction. Therefore, $\alpha(\mathbb{AG}_N(R))\leq \alpha(\mathbb{AG}_N(R_S))$, as desired.
}
\end{proof}

Finally as an application of the nil-graph of ideals in the ring theory we have the following corollary which shows that number of minimal prime ideals of a Noetherian reduced ring coincides number of maximal ideals of the total ring of $R$.

\begin{cor}
Let $R$ be a Noetherian reduced ring. Then $$|{\rm Min}(R)|=|{\rm Max}(T(R))|=\log_2(\alpha(\mathbb{AG}_N(R))).$$
\end{cor}
\begin{proof}
{Setting ${\rm Min}(R)=\{\mathfrak{p}_1,\mathfrak{p}_2,\ldots, \mathfrak{p}_n\}$ and $S=R\setminus \bigcup_{\mathfrak{p}\in {\rm Min}(R)}\mathfrak{p}$,  we have $T(R)\cong R_{{\mathfrak{p}}_1}\times\cdots\times R_{{\mathfrak{p}}_n}$, by Lemma \ref{matlis}. Since every $R_{{\mathfrak{p}}_i}$ is a field, Corollary \ref{artinianalpha} and Theorem \ref{reducedfinitemin} imply that $2^{|{\rm Min}(R)|-1}=2^{|{\rm Max}(T(R))|-1}=\alpha(\mathbb{AG}_N(R))$. So, the assertion follows.
}
\end{proof}


\vspace{5mm} \noindent{\bf\large 4. The Genus of Nil-Graphs of Ideals}

In \cite[Corollary 2.11]{behboodigenus}, it is proved that for integers $q > 0$ and $g\geq0$, there are finitely many Artinian rings
$R$ satisfying the following conditions:\\
(1) $\gamma(\mathbb{AG}(R)) = g$,\\
(2) $|\frac{R}{\mathfrak{m}}|\leq q$ for any maximal ideal $\mathfrak{m}$ of $R$.

We begin this section with a similar result for the nil-graph of ideals.

\begin{thm}\label{genus}
Let $g$ and $q>0$ be non-negative integers. Then there are finitely many Artinian rings $R$ such that $\gamma(\mathbb{AG}_N(R))=g$ and $|\frac{R}{\mathfrak{m}}|\leq q$, for every maximal ideal $\mathfrak{m}$ of $R$.
\end{thm}
\begin{proof}
{Let $R$ be an Artinian ring. Then \cite[Theorem 8.7]{ati} implies that $R\cong R_1\times\cdots\times R_n$, where $n$ is a positive integer and each $R_i$ is an Artinian local ring. We claim that for every $i$, $|R_i|\leq q^{\mathbb{I}(R_i)}$. Since $\gamma(\mathbb{AG}_N(R))<\infty$, we deduce that $\gamma(\mathbb{AG}_N(R_i))<\infty$, for every $i$. So by Remark \ref{niladjacent} and formula for the genus of complete graphs, every $R_i$ has finitely many ideals. Therefore, by hypothesis and \cite[Lemma 2.9]{behboodigenus}, we have $|R_i|\leq |\frac{R_i}{\mathfrak{m}_i}|^{\mathbb{I}(R_i)}\leq q^{\mathbb{I}(R_i)}$ and so the claim is proved. To complete the proof, it is sufficient to show that $|R|$ is bounded by a constant, depending only on $g$ and $q$. With no loss of generality, suppose that $|R_1|\geq |R_i|$, for every $i\geq 2$. By the formula for the genus of complete graphs, $\frac{\mathbb{I}(R_1)-5}{12}\leq \gamma(\mathbb{AG}_N(R_1))\leq g$. Hence $|\mathbb{I}(R_1)|\leq  12g+5$ and so
$$|R|\leq |R_1|^n\leq (q^{\mathbb{I}(R_1)})^n\leq q^{n(12g+5)}.$$ So, we are done.}
\end{proof}

 Let $\{R_i\}_{i\in\mathbb{N}}$ be an infinite family of Artinian rings such that every $R_i$ is a direct product of $4$ fields. Then it is clear that $\gamma(\mathbb{AG}_N(R_i))=1$, for every $i$. So, the condition $|\frac{R}{\mathfrak{m}}|\leq q$, for every maximal ideal $\mathfrak{m}$ of $R$, in the previous theorem is necessary.

 Let $R$ be a Noetherian ring. Then one may ask does $\gamma(\mathbb{AG}_N(R))<\infty$ imply that $R$ is Artinian?
The answer of this question is negative. To see this, let $R\cong S\times D$, where $S$ is a ring with at most one non-trivial ideal and $D$ is a Noetherian integral domain which is not a field. Then it is easy to check that $\mathbb{AG}_N(R)$ is a planar graph and $R$ is a Noetherian ring which is not Artinian.

 Before proving the next lemma, we need the following notation. Let $G$ be a graph and $V'$ be the set of vertices of $G$ whose
degrees equal one. We use $\widetilde{G}$ for the subgraph
$G\setminus V'$ and call it the {\it reduction} of $G$.

\begin{lem}\label{reduction}
$\gamma(G)=\gamma(\widetilde{G})$, where $\widetilde{G}$ is the
reduction of $G$.
\end{lem}

\begin{remark}\label{face}
It is well-known that if $G$ is a connected graph of genus $g$,
with $n$ vertices, $m$ edges and $f$ faces, then $n-m+f=2-2g$.
\end{remark}

In the following, all Artinian rings, whose nil-graphs of ideals have genus at most one, are classified.
\begin{thm}\label{genuslesssthantwo}
Let $R$ be an Artinian ring. If $\gamma(\mathbb{AG}_N(R))<2$,
then $|{\rm Max}(R)|\leq 4$ and moreover, the following statements
hold.
\begin{enumerate}
\item[\rm(i)] If $|{\rm Max}(R)|= 4$, then $\gamma(\mathbb{AG}_N(R))<2$ if and only if $R$ is isomorphic to a direct product of four fields.
 \item[\rm(ii)] If $|{\rm Max}(R)|= 3$, then $\gamma(\mathbb{AG}_N(R))<2$ if and only if
$R\cong F_1\times F_2\times R_3$, where $F_1, F_2$ are fields and
$R_3$ is an Artinian local ring with at most two non-trivial
ideals.
 \item[\rm(iii)] If $|{\rm Max}(R)|= 2$, then $\gamma(\mathbb{AG}_N(R))<2$ if and only if
either $R\cong F_1\times R_2$, where $F_1$ is a field and $R_2$
is an Artinian local ring with at most three non-trivial ideals or
$R\cong R_1\times R_2$, where every $R_i$ {\rm(}$i=1,2${\rm)} is an
Artinian local ring with at most one non-trivial ideal.
\item[\rm(iv)] If $R$ is local, then $\gamma(\mathbb{AG}_N(R))<2$ if and only if $R$ has at most $7$ non-trivial ideals.
\end{enumerate}
\end{thm}

\begin{proof}
{Let $\gamma(\mathbb{AG}_N(R))<2$. First we show that $|{\rm
Max}(R)|\leq 4$. Suppose to the contrary, $|{\rm Max}(R)|\geq 5$.
By \cite[Theorem 8.7]{ati}, $R\cong R_1 \times\cdots\times R_5$,
where every $R_i$ is an Artinian ring. Let $$I_1=R_1\times
(0)\times(0)\times (0)\times (0);\hspace{1cm} I_2=(0)\times
R_2\times (0)\times (0)\times (0);$$
$$I_3=R_1\times
R_2\times(0)\times (0)\times (0);\hspace{1cm} J_1=(0)\times
(0)\times R_3\times (0)\times (0);$$
$$J_2=(0)\times (0)\times (0)\times
R_4\times (0);\hspace{1cm} J_3=(0)\times (0)\times(0)\times
(0)\times R_5;$$
$$J_4=(0)\times (0)\times R_3\times
R_4\times(0);\hspace{1cm} J_5=(0)\times (0)\times (0)\times
R_4\times R_5;$$
$$J_6=(0)\times (0)\times R_3\times (0)\times
R_5;\hspace{1cm} J_7=(0)\times (0)\times R_3\times R_4\times
R_5.$$ Then for every $1\leq i\leq 3$ and every $1\leq j\leq 7$,
$I_i$ and $J_j$ are adjacent and so $K_{3,7}$ is a subgraph of
$\mathbb{AG}_N(R)$. Thus by the formula for the genus of the
complete bipartite graph, we have
$\gamma(\mathbb{AG}_N(R))\geq\gamma(K_{3,7})\geq 2$, a
contradiction.

(i) Let $|{\rm Max}(R)|=4$ and $\gamma(\mathbb{AG}_N(R))<2$. By
\cite[Theorem 8.7]{ati}, $R\cong R_1\times R_2\times R_3\times
R_4$, where every $R_i$ is an Artinian local ring. We show that
every $R_i$ is a field. Suppose not and with no loss of
generality, $R_4$ contains a non-trivial ideal, say
$\mathfrak{a}$. Set
$$I_1=R_1\times
(0)\times(0)\times (0);\ I_2=(0)\times R_2\times (0)\times (0);\
I_3=R_1\times R_2\times(0)\times (0);$$
$$I_4=R_1\times R_2\times(0)\times \mathfrak{a};\ J_1=(0)\times (0)\times R_3\times
(0);\ J_2=(0)\times (0)\times (0)\times R_4;$$
$$J_3=(0)\times
(0)\times(0)\times \mathfrak{a};\ J_4=(0)\times (0)\times
R_3\times R_4;\ J_5=(0)\times (0)\times R_3\times \mathfrak{a}.$$
It is clear that every $I_i$, $1\leq i\leq 4$, is adjacent to
$J_j$, $1\leq j\leq 5$, and so $K_{4,5}$ is a subgraph of
$\mathbb{AG}_N(R)$. Thus by the formula for the genus of the
complete bipartite graph, we have
$\gamma(\mathbb{AG}_N(R))\geq\gamma(K_{4,5})\geq 2$, a
contradiction. Conversely, assume that $R\cong F_1\times
F_2\times F_3\times F_4$, where every $F_i$ is a field. We show
that $\gamma(\mathbb{AG}_N(R))=1$. By Lemma \ref{reduction}, it is
enough to prove that $\gamma(\widetilde{\mathbb{AG}_N(R)})=1$. We
know that $\widetilde{\mathbb{AG}_N(R)}$ has $4$ vertices of
degree $6$ and $6$ vertices of degree $3$. So,
$\widetilde{\mathbb{AG}_N(R)}$ has $n=10$ vertices and $m=21$
edges. Also, it is not hard to check that
$\widetilde{\mathbb{AG}_N(R)}$ has $f=11$ faces. Now, Remark
\ref{face} implies that $\gamma(\widetilde{\mathbb{AG}_N(R)})=1$.

(ii) Let $\gamma(\mathbb{AG}_N(R))<2$ and $R\cong R_1\times
R_2\times R_3$, where every $R_i$ is an Artinian local ring. We
show that at least two of the three rings $R_1$, $R_2$ and $R_3$
are fields. Suppose not and with no loss of generality, $\mathfrak{b}$ and
$\mathfrak{c}$ are non-trivial ideals of $R_2$ and $R_3$, respectively. Set
$$I_1=R_1\times
(0)\times(0);\ I_2=(0)\times R_2\times (0);\ I_3=R_1\times
R_2\times(0); \ I_4=R_1\times \mathfrak{b}\times(0);$$
$$J_1=(0)\times
\mathfrak{b}\times (0);\ J_2=(0)\times (0)\times \mathfrak{c};\ J_3=(0)\times
\mathfrak{b}\times \mathfrak{c};$$
$$ J_4=(0)\times (0)\times R_3;\ J_5=(0)\times
\mathfrak{b}\times R_3.$$ It is clear that every $I_i$, $1\leq i\leq 4$,
is adjacent to $J_j$, $1\leq j\leq 5$, and so $K_{4,5}$ is a
subgraph of $\mathbb{AG}_N(R)$. Thus by the formula for the genus
of the complete bipartite graph, we have
$\gamma(\mathbb{AG}_N(R))\geq\gamma(K_{4,5})\geq 2$, a
contradiction. Thus with no loss of generality, we can suppose
that $R\cong F_1\times F_2\times R_3$, where $F_1$ and $F_2$ are
fields and $R_3$ is an Artinian local ring. Now, we prove that
$R_3$ has at most two non-trivial ideals. Suppose to the contrary,
$\mathfrak{a}$, $\mathfrak{b}$ and $\mathfrak{c}$ are three
distinct non-trivial ideals of $R_3$. Let
$$I_1=(0)\times (0) \times R_3;\ I_2=(0)\times (0) \times \mathfrak{a};\ I_3=(0)\times (0) \times \mathfrak{b};\
I_4=(0)\times (0) \times \mathfrak{c};$$
$$J_1=F_1\times (0)\times (0);\ J_2=(0)\times F_2\times (0);\ J_3=F_1\times F_2\times (0);$$
$$J_4=F_1\times F_2\times \mathfrak{a};\ J_5=F_1\times F_2\times \mathfrak{b};\ J_6=F_1\times F_2\times \mathfrak{c}.$$
Clearly, every $I_i$, $1\leq i\leq 4$, is adjacent to $J_j$,
$1\leq j\leq 6$, and so $K_{4,6}$ is a subgraph of
$\mathbb{AG}_N(R)$. Thus by the formula for the genus of the
complete bipartite graph, we have
$\gamma(\mathbb{AG}_N(R))\geq\gamma(K_{4,6})\geq 2$, a
contradiction. Conversely, let $R\cong F_1\times F_2\times R_3$,
where $F_1$ and $F_2$ are fields and $R_3$ be a ring with two
non-trivial ideals $\mathfrak{c}$ and $\mathfrak{c}'$. Set
$$I_1=(0)\times (0) \times R_3;\ I_2=(0)\times (0) \times \mathfrak{c};\ I_3=(0)\times (0) \times \mathfrak{c}';$$
$$J_1=F_1\times (0)\times (0);\ J_2=(0)\times F_2\times (0);\ J_3=F_1\times F_2\times (0).$$
Then for every $1\leq i,j\leq 3$, we have $I_iJ_j=(0)$. Hence
$\gamma(\mathbb{AG}_N(R))\geq\gamma(K_{3,3})\geq 1$. However, in
this case, $\mathbb{AG}_N(R)$ is a subgraph of
$\mathbb{AG}_N(F_1\times F_2\times F_3\times F_4)$ (in which
every $F_i$ is a field). Therefore, by (i),
$\gamma(\mathbb{AG}_N(R))=1$. If $R_3$ contains at most one
non-trivial ideal, then it is not hard to check that
$\mathbb{AG}_N(R)$ is a planar graph. This completes the proof of
(ii).

(iii) Assume that $\gamma(\mathbb{AG}_N(R))<2$ and $R\cong
R_1\times R_2$, where $R_1$ and $R_2$ are Artinian local rings.
We prove the assertion in the following two cases:

Case 1. $R\cong F_1\times R_2$, where $F_1$ is a field and $R_2$
is an Artinian local ring. In this case, we show that $R_2$ has
at most three non-trivial ideals. Suppose to the contrary, $R_2$
has at least four non-trivial ideals. Then for every two non-zero
ideals $I_2\neq R_2$ and $J_2$ of $R_2$, the vertices $F_1\times
I_2$ and $(0)\times J_2$ are adjacent and so $K_{4,5}$ is a
subgraph of $\mathbb{AG}_N(R)$. Thus by the formula for the genus
of the complete bipartite graph, we have
$\gamma(\mathbb{AG}_N(R))\geq\gamma(K_{4,5})\geq 2$, a
contradiction.

 Case 2. Neither $R_1$ nor $R_2$ is a field. We prove that every
$R_i$ has at most one non-trivial ideal. Suppose not and with no
loss of generality, $R_2$ has two distinct non-trivial ideals.
Then every ideal of the form $R_1\times J$ is adjacent to every
ideal of the form $I\times K$, where $I$ and $J$ are proper ideals
of $R_1$ and $R_2$, respectively, and $K$ is an arbitrary ideal of
$R_2$. So $\gamma(\mathbb{AG}_N(R))\geq\gamma(K_{3,7})\geq 2$, a
contradiction.

 Conversely, if $R\cong F_1\times R_2$, where $F_1$ is a field and
$R_2$ is an Artinian local ring with $n\leq 3$ non-trivial
ideals, then one can easily show that
$$\gamma(\mathbb{AG}_N(R))=
\begin{cases}
1;   & n=2,3\\
0;    & n=1.
\end{cases}
$$
 Now, suppose that  $R\cong R_1\times R_2$, where $R_1$ and $R_2$
are Artinian local rings with one non-trivial ideals. Then it is
not hard to show that $\gamma(\mathbb{AG}_N(R))=1$. This comletes
the proof of (iii).

(iv) This follows from the formula of genus for the complete
graphs and Remark \ref{niladjacent}.}
\end{proof}
From the proof of the previous theorem, we have the following immediate corollary.
\begin{cor}\label{planarity}
Let $R$ be an Artinian ring. Then $\mathbb{AG}_N(R)$ is a planar
graph if and only if $|{\rm Max}(R)|\leq 3$ and $R$ satisfies one
of the following conditions:
\begin{enumerate}
 \item[\rm(i)] $R$ is isomorphic to the direct product of three fields.
 \item[\rm(ii)] $R\cong F_1\times R_2$, where $F_1$ is a field and $R_2$ is an Artinian local ring
with at most one non-trivial ideal.
 \item[\rm(iii)] $R$ is a local ring with at most four non-trivial ideals.
\end{enumerate}
\end{cor}

 We close this paper with the following example.

\begin{example}\label{ex2}
\rm(i) Suppose that $R\cong\frac{\mathbb{Z}_6[x]}{(x^m)}$, where $m\geq 2$. Let $I_1=(3)$, $I_2=(3x)$, $I_3=(3x+3)$, $J_1=(2)$, $J_2=(4)$, $J_3=(2x)$, $J_4=(4x)$, $J_5=(2x+2)$, $J_6=(4x+2)$ and $J_7=(2x+4)$. Then one can check that these ideals are distinct vertices of $\mathbb{AG}_N(R)$. Also, every $I_i$ ($1\leq i\leq 3$) is adjacent to every $J_k$ ($1\leq k\leq 7$). Thus $K_{3,7}$ is a subgraph of $\mathbb{AG}_N(R)$ and so the formula of genus for the complete bipartite graphs implies that $\gamma(\mathbb{AG}_N(R))\geq 2$.\\
\rm(ii) Let $R\cong\frac{\mathbb{Z}_4[x]}{(x^3)}$. Set $I_1=(2x)$, $I_2=(2x^2)$, $I_3=(2x+2x^2)$, $J_1=(2)$, $J_2=(2+x^2)$, $J_3=(2+2x^2)$, $J_4=(2-x^2)$, $J_5=(2+2x)$, $J_6=(2+2x+x^2)$ and $J_7=(2+2x+2x^2)$. Similar to (i), one can show that every $I_i$ is adjacent to every $J_k$ and so $\gamma(\mathbb{AG}_N(R))\geq 2$.
\end{example}

{}

\end{document}